\documentclass[11pt]{amsart}
 \usepackage[margin=3.8cm]{geometry}
\numberwithin{equation}{section}
\usepackage{amssymb}
\usepackage{amsmath, amsfonts,amsthm,amssymb,amscd, verbatim,graphicx,color,multirow,booktabs, caption,tikz,tikz-cd, mathdots,bm}
\usepackage{tikz-cd}
 \usepackage{hyperref} 
\usetikzlibrary{positioning}

\newtheorem*{theorem*}{Theorem}
\newtheorem{theorem}{Theorem}[section]
\newtheorem{lemma}[theorem]{Lemma}

\newtheorem{proposition}[theorem]{Proposition}
 \newtheorem{corollary}[theorem]{Corollary}

      \theoremstyle{definition}

     \theoremstyle{remark}
     \newtheorem{remark}[theorem]{Remark}
     
  \newcommand{\Sf}{\mathcal{S}} 
\newcommand{\diam}{\mathop{\mathrm{diam}}} 
\newcommand{\Sym}{\mathop{\mathrm{Sym}}}
\newcommand{\Sch}{\mathop{\mathrm{Sch}}}
\newcommand{\Cay}{\mathop{\mathrm{Cay}}}

 \definecolor{mycolor}{rgb}{0.55,0.0,0.16}
  \definecolor{myred}{rgb}{0.6,0.0,0.16}
  \definecolor{mygreen}{rgb}{0.0,0.6,0.16}
  \definecolor{myviolet}{rgb}{1,0,1}

\hypersetup{
colorlinks=true,
linkcolor=true,
linktocpage=true,
pageanchor=true,
}

 \AtBeginDocument{
     \hypersetup{
  linkcolor=myred,
  urlcolor=mycolor,
citecolor=mygreen
}
     }

\makeatletter
\@namedef{subjclassname@2020}{%
  \textup{2020} Mathematics Subject Classification}
\makeatother


\title[The diameter of random Schreier graphs]{The diameter of random Schreier graphs} 

\author{Daniele Dona}
\address{Hun-Ren Alfr\'ed R\'enyi Institute of Mathematics, Budapest, Hungary}
\email{dona@renyi.hu}

\author{Luca Sabatini}
\address{Luca Sabatini, University of Warwick} 
\email{sabatini.math@gmail.com}

\begin{document}

\begin{abstract} 
We give a combinatorial proof of the following theorem.
Let $G$ be any finite group acting transitively on a set of cardinality $n$.
If $S \subseteq G$ is a random set of size $k$,
with $k \geq (\log n)^{1+\varepsilon}$ for some $\varepsilon >0$,
then the diameter of the corresponding Schreier graph is $O(\log_k n)$ with high probability.
Except for the implicit constant, this result is the best possible.
\end{abstract} 

\maketitle

\section{Introduction}

Let $G$ be a finite group acting transitively on a set $\Omega$, and let $S \subseteq G$.
The {\itshape Schreier graph} $\Gamma=\Sch(G \circlearrowleft \Omega,S)$ is the directed multigraph with vertex set $V(\Gamma)= \Omega$,
and the edges are the pairs $(\omega,\omega^s)$ for every $\omega \in \Omega$ and $s \in S$.
This is a natural generalization of Cayley graphs, which arise when $G$ is a regular permutation group on $\Omega$,
and we can identify $\Omega$ with $G$.
The general setting is much wilder, as there is no natural action of $G$ on $V(\Gamma)$ that preserves the graph structure.
For example, it is well known that every regular graph of even degree is a Schreier graph over $\Sym(\Omega)$ \cite{Gro77}.

In this paper, we are interested in the diameter of random Schreier graphs,
where by random we mean that the action of $G$ on $\Omega$ is fixed, and a multiset $S \subseteq G$ of size $k$ is chosen uniformly at random.
The value of $k$ is allowed to depend on the size of $\Omega$, which in turn is going to infinity.
It is clear that the results have to depend on the particular group $G$ and action,
and on how large $k$ is with respect to the cardinalities of $G$ and $\Omega$.
For specific ``nice'' actions and bounded $k$, strong results have been obtained in \cite{BF82,FJRST98,BGGT15,EJ22}.
For example, for $\Sym(\Omega)$ in its natural action on $r$-tuples ($r$ fixed), and for any fixed $k \geq 2$,
the symmetrization of the resulting Schreier graph is a $2k$-regular expander graph with probability tending to $1$ as $|\Omega|\rightarrow\infty$ \cite{FJRST98}.
The situation turned out to be quite different for an arbitrary ``bad'' action.
Since $\Gamma$ is connected if and only if the subgroup generated by $S$ is transitive,
it is necessary for $k$ to grow with the size of $\Omega$:
for a regular permutation group, $k$ has to be at least the minimal number of generators.

Let $n$ be the cardinality of $\Omega$.
For $k = \Omega(\log n)$,
a logarithmic bound $O(\log n)$ for the diameter of a random Cayley graph was first given by Alon, Barak and Manber \cite{ABM87}
using a combinatorial method.
Nowadays, we know that $\Omega(\log n)$ random elements are always sufficient to get an expander with high probability \cite{AR94,Sab22}.
For slightly larger $k$, Roichman \cite[Theorem 5.6]{Roi96} proved the following better diameter bound for Cayley graphs:

\begin{theorem*}[Roichman] 
Let $\varepsilon >0$, and let $G$ be any group of cardinality $n$.
Let $k= \lfloor (\log n)^{1+\varepsilon} \rfloor$, and let $S \subseteq G$ be a random set of $k$ elements.
Then
$$
\diam ( \> \Cay(G,S) \> ) 
\> \leq \> 
2(1+\varepsilon^{-1}) \> \frac{\log n}{\log k} 
$$
with probability $1-o(1)$ as $n\rightarrow\infty$.
\end{theorem*} 

Roichman's theorem is sharp in a number of directions.
First, it is clear that a bound of the type $O(\log_k n)$ is the best possible for the diameter of a $k$-regular graph on $n$ vertices.
Second, it is not hard to see \cite[Proposition 5.7]{Roi96} that, for an abelian group of order $n$, $(\log n)^{1+\varepsilon}$
elements are necessary to get such a bound $O(\log_k n)$.
Third, a constant of the type $O(\varepsilon^{-1})$ is the best possible for small $\varepsilon$ (we need to write $1+\varepsilon^{-1}$ rather than just $\varepsilon^{-1}$ for the bound to be true with $\varepsilon$ large).

It is natural to ask for a generalization to arbitrary transitive groups.
It is remarkable that Roichman obtains his theorem by studying the mixing time of random random walks (see also \cite{DH96}).
Besides being a tortuous road to give an estimate for the diameter,
it is not clear whether these ideas can be applied in the general framework:
they use many distinguished properties of Cayley graphs, in particular their vertex-transitivity.
Moreover, while it is natural that abelian groups represent the worst case for Cayley graphs,
this is not obvious in general: an abelian group only provides vertex-transitive Schreier graphs,
which is a very desirable property when bounding the directed diameter \cite{Bab06,Skr23}.
In this paper, we come back to combinatorics.
We provide a direct proof of the following theorem,
which settles the question of the diameter of random Schreier graphs up to the multiplicative constant.

\begin{theorem} \label{th:main}
For every $\varepsilon >0$ there exists $C>0$ such that the following holds.
Let $\Omega$ be a set of cardinality $n$, and let $G$ be any finite group acting transitively on $\Omega$.
Let $(\log n)^{1+\varepsilon}\leq k\leq n$,
and let $S \subseteq G$ be a random multiset of $k$ elements.
Then
$$
\diam ( \> \Sch(G\circlearrowleft\Omega,S) \> ) 
\> \leq \> 
C \> \frac{\log n}{\log k} 
$$
with probability $1-o(1)$ as $n\rightarrow\infty$.
\end{theorem} 

The diameter of a multigraph equals the diameter of the underlying graph,
so Theorem \ref{th:main} implies the corresponding version for random sets.
(The distribution of random $k$-multisets is not the same as the distribution of random $k$-sets, but this is only a minor technical inconvenience.)

Our proof yields $C=O(\varepsilon^{-3})$, but can be modified to yield $C=O_{\delta}(\varepsilon^{-1-\delta})$ for any $\delta>0$: see Remark~\ref{re:cepsilon} below. The special case of Cayley graphs already shows that one cannot do better than $C=O(\varepsilon^{-1})$, so $k$ in Theorem~\ref{th:main} cannot be extended to slower functions like $k=\log n\log\log n$. We also highlight the validity of Theorem~\ref{th:main} for all large $k$,
for example $k= \lfloor n^\delta \rfloor$ with fixed $\delta>0$.
In this case, we obtain that the diameter is bounded by $C \delta^{-1}$, where $C$ is an absolute constant.

\begin{remark}
The $k$ random elements in Theorem \ref{th:main} are not able to generate the group $G$ in general,
as there are transitive permutation groups of degree $n$ that require up to roughly $n/\sqrt{\log n}$ generators \cite{KN88,LMM00}.
In fact, the result in \cite{Sab22} implies that $O(\log n)$ random elements generate a {\itshape transitive subgroup},
with probability $1-o(1)$ as $n\rightarrow\infty$,
and that in that case the diameter bound is $O(\log n)$. Theorem \ref{th:main} shows that $(\log n)^{1+\varepsilon}$ random elements provide the optimal generation almost surely, in the sense that the bound is the stronger $O(\log_{k}n)$.
\end{remark}

\vspace{0.1cm}
\section{Growth in random Schreier graphs} \label{sec2} 

The proof of Theorem \ref{th:main} is inspired by the more modern philosophy of growth in groups.
We refer the reader to the beautiful survey \cite{Hel15} for an introduction to this vibrant subject,
and to \cite{Mur19} for some generalizations in the context of group actions (i.e. Schreier graphs).
We believe that this article is yet another demonstration of the power of these techniques.
We follow two main steps, which we explain in Propositions \ref{pr1} and \ref{pr2} below.
When $A \subseteq G$ and $t \in \mathbb{N}$,
we write $A^t = \{ a_1 \cdot \ldots \cdot a_t : a_1, \ldots, a_t \in A\}$ for the set of products of $t$ elements of $A$.
 If $\omega \in \Omega$, 
 we write $\Sf_A(\omega,t)$ to denote the sphere of radius $t$ in $\Sch(G \circlearrowleft \Omega,A)$ and centered in $\omega$, i.e.
$$ \Sf_A(\omega,t) \> = \> \omega^{(A^{t})}
\> = \> \{\omega^{g}:g\in A^{t}\} . $$
We remark that $\Sf_A(\omega,t) = \Sf_{A^t}(\omega,1)$ for all $t \in \mathbb{N}$.
Moreover,
\begin{equation} \label{eqDiam}
\diam( \> \Sch(G \circlearrowleft \Omega,A) \> ) \> = \> \max_{\omega \in \Omega} \min \{ t \in \mathbb{N} : \Sf_A(\omega,t) = \Omega \} . 
\end{equation} 
We write $A \sim \mu_G(k)$ to say that $A \subseteq G$ is a random multiset of size $k$ chosen with the uniform distribution.

\begin{proposition}[Growth] \label{pr1}
For every $\varepsilon >0$ there exists $C>0$ such that the following holds.
Let $G$ be any finite group acting transitively on a set $\Omega$ with $|\Omega|=n$,
and let $(\log n)^{1+\varepsilon}\leq k\leq n$.
Fix $\omega \in \Omega$. Then, for $A\sim \mu_{G}(k)$, we have
$$ \left|\Sf_A\left(\omega, \left\lfloor C \log_k n \right\rfloor \right)\right| 
\> \geq \>
 \frac{n}{k^{\varepsilon/2}} $$
with probability $1-o(1)$ as $n\rightarrow\infty$.
\end{proposition} 

At this stage, we stress that the multiplicative gap between $|\mathcal{S}_A(\omega,t)|$ and $|\Omega|$ depends on $k$
(and is unbounded).
Once we have a large sphere, we bound the diameter with the help of a few more random elements.

\begin{proposition}[Diameter bound] \label{pr2}
Let $\delta>0$, and let $G,\Omega,n,k$ be as above.
Suppose that, for some $\omega \in \Omega$,
for $A \sim \mu_G(k)$ we have $|\Sf_A(\omega,t)| \geq |\Omega|/k^\delta$
with probability $1-o(1)$ as $n\rightarrow\infty$.
Then, for $B \sim \mu_G(2k + 8k^\delta \log n)$, we have
$$ \diam( \> \Sch(G \circlearrowleft \Omega, B) \> ) \> \leq \> 2t+2 $$
with probability $1-o(1)$ as $n\rightarrow\infty$.
\end{proposition} 

In order to use Proposition \ref{pr2} to prove Theorem \ref{th:main}, it is important to notice that $k^\delta \log n$ is comparable with $k$,
if $k$ is larger than $(\log n)^{1+\varepsilon}$ and $\delta$ is small enough with respect to $\varepsilon$.
In the remaining part of this section we prove Proposition \ref{pr1},
while Proposition \ref{pr2} and Theorem \ref{th:main} will be proven in Section \ref{se:fill}.

\subsection{Random conjugates} 

Let $X,Y \subseteq \Omega$, and let $g \in G$ be random.
It is intuitive that, if $\Omega$ is large, then a random conjugate $X^g$ of $X$ intersects $Y$ in a few points with high probability.
This is the content of the following lemma.

\begin{lemma} \label{le:doublecount} 
Let $G$ be any finite group acting transitively on a set $\Omega$.
          Let $r,s>0$ and $X,Y \subseteq \Omega$.
          Suppose that for at least $|G|/s$ of the elements $g \in G$ we have $|X^g \cap Y| \geq r$.
          Then $|\Omega| \leq \frac{|X||Y|s}{r}$.
\end{lemma}
\begin{proof} 
          Let $t$ be the cardinality of a stabilizer, namely $t=|G|/|\Omega|$.
         Let $y \in Y$.
         For each $x \in X$, there are precisely $t$ elements $g \in G$ such that $y=x^g$.
         Therefore, each element of $Y$ lies in the same number $t|X|$ of translates $(X^g)_{g \in G}$, and so
          $$ |\{ (y,g) \in Y \times G :  y \in X^g \}| \> = \> t|X||Y| . $$
          On the other hand, we have
        $$
    |\{ (y,g) \in Y \times G :  y \in X^g \}| 
    \> = \>
    \sum_{g \in G} |X^g \cap Y|
     \> \geq \>
      \frac{|G|r}{s} .
    $$
    The proof follows.
\end{proof} 

We can read Lemma \ref{le:doublecount} as follows. If
        \begin{equation} \label{eqCond2}
            |Y| \> < \> \frac{|\Omega|r}{|X|s} ,
        \end{equation}
       then there are at least $|G|(1-s^{-1})$ elements $g \in G$ such that $|X^g \cap Y|<r$.
       The next key result exploits this fact to give a glimpse of growth.
       For a random distribution $\mu$,
       we write $\mathbb{P}_{\mu}$ to denote the probability of a certain event with respect to $\mu$.

\begin{lemma}[One-step growth] \label{le:OneStep}
Let $G$ be any finite group acting transitively on a set $\Omega$.
Let $r,s >0$ and $X \subseteq \Omega$,
and let $k$ be an integer such that
        \begin{equation} \label{eqCond1} 
            |\Omega| \> \geq \> 
    \frac{k s |X|^2 }{r} . 
        \end{equation}
Then, for $A\sim \mu_{G}(k)$, we have
$$
\mathbb{P}_{\mu_G(k)} \left( \> \left| X^A \right| \geq k \left( |X| - r \right) \> \right)
\> \geq \> 
\left( 1- \frac{1}{s} \right)^k.
$$
\end{lemma}
\begin{proof}
  We work by induction on $k$, the case $k=1$ being trivial.
    Let $k \geq 2$, and suppose that (\ref{eqCond1}) holds.
    By induction, for $A'\sim\mu_{G}(k-1)$, we have
\begin{equation}\label{eq:muk-1}
\mathbb{P}_{\mu_{G}(k-1)}\left( \left| X^{A'} \right|
\> \geq \>
 (k-1)\left(|X|-r\right)\right)
 \> \geq \>
 \left( 1- \frac{1}{s} \right)^{k-1} .
\end{equation}
    For any evaluation $(g_1,\ldots,g_{k-1})$ of $A'$, we set
$$ Y \> = \> X^{\{g_1, \ldots ,g_{k-1}\}} \> = \> \bigcup_{i=1}^{k-1} X^{g_i} . $$
Certainly $|Y| < k|X|$, and so it is easy to check that (\ref{eqCond2}) follows from (\ref{eqCond1}).
Take a random $g_{k}\in G$, which is the same as considering $\{g_{k}\}\sim\mu_{G}(1)$.
We obtain
\begin{equation}\label{eq:mu1}
\mathbb{P}_{\mu_{G}(1)} \biggl( |X^{g_k} \cup Y| 
\> \geq \>
 |X| + |Y| - r \biggr)
  \> \geq \> 
  1- \frac{1}{s} .
\end{equation}
Since $A'\sim\mu_{G}(k-1)$ and $\{g_{k}\}\sim\mu_{G}(1)$, we have $A=A'\cup\{g_{k}\}\sim\mu_{G}(k)$.
Putting together (\ref{eq:muk-1}) and (\ref{eq:mu1}),
    with probability at least $(1-s^{-1})^k$,
    we have
\begin{align*}
	\left| X^A \right| & \> = \> 
	\left| X^{A'} \cup X^{g_k} \right| \\ & \> \geq \> 
	(k-1)\left(|X|-r\right) + |X| - r \\ & \> = \> 
  	|X|k -rk \> 
	\end{align*} 
    as desired.
          \end{proof} 

         For the convenience of the reader, we state a special version of Lemma~\ref{le:OneStep}, which is the one we will use later.
         
         \begin{corollary} \label{co:Explicit} 
         Let $d \geq 2$.
Let $G$ be any finite group acting transitively on a set $\Omega$.
Let $X \subseteq \Omega$, and let $k\geq 4$ be an integer such that
\begin{equation}\label{eq:explcondition}
|X| \> \leq \> \frac{|\Omega|}{k^{d}} .
\end{equation}
Then, for $A\sim \mu_{G}(k)$, we have
$$ \mathbb{P}_{\mu_G(k)}\left( \> \left|X^A \right| \geq \sqrt{k}|X| \> \right) 
\> \geq \> 
\left( 1 - \frac{2}{k^{d-1}} \right)^k . $$
\end{corollary} 
\begin{proof}
Apply Lemma \ref{le:OneStep}, with
\begin{align*}
r & = |X|\left( 1 - \frac{1}{\sqrt{k}} \right), 
&
s & = \frac{k^{d-1}}{2} ,
\end{align*}
and note that $r\geq|X|/2$ by our choice of $k$.
\end{proof}

\subsection{Growth of spheres} 

We start estimating the size of random spheres.
To do so, we divide the random set $A$ into a bounded number of pieces of comparable size,
and use the trivial observation that, if $A_1,A_2 \subseteq A$, then $A_1A_2 \subseteq A^2$.
This allows to create apparently new random elements while we are growing.

\begin{lemma} \label{lemGS} 
For every $\varepsilon \in (0,1/2)$ there exists $C >0$ such that the following holds.
Let $G$ be any finite group acting transitively on a set $\Omega$ with $|\Omega|=n$,
and let $(\log n)^{1+\varepsilon}\leq k\leq n$.
Fix $\omega \in \Omega$.
Then, for $A\sim \mu_{G}(k)$, we have
$$
\left|\Sf_A\left(\omega, \left\lfloor C \log_k n \right\rfloor \right)\right| \> \geq \> \frac{n}{k^{2/\varepsilon}}
$$
with probability $1-o(1)$ as $n\rightarrow\infty$.
\end{lemma} 
\begin{proof}
Let $D,h \geq 1$ be positive integers, to be chosen later, such that $Dh \leq k$.
We can rewrite $A \subseteq G$ as
\begin{align*}
A & =A_1\cup\ldots\cup A_{D}\cup A', & A_{i} & \sim\mu_{G}(h), & A' & \sim\mu_{G}(k-Dh) ,
\end{align*}
where $A_1, \ldots ,A_D,A'$ are disjoint as multisets.
Remark that $\prod_{i=1}^{D} A_i \subseteq A^{D}$, so we can effectively work with each $A_{i}$ independently, and discard $A'$.
We apply Corollary~\ref{co:Explicit} iteratively,
choosing $X=\Sf_{A_1 \cdot \ldots \cdot A_{i-1}}(\omega,1)$, $A_{i}\sim\mu_{G}(h)$, $d=2 \varepsilon^{-1}$,
starting with $X=\{\omega\}$ for $i=1$.
Observe that $X^{A_{i}} = \Sf_{A_1 \cdot \ldots \cdot A_{i}}(\omega,1)$ at each step.
If (\ref{eq:explcondition}) is true for all $1\leq i < D$, then
$$
|\Sf_{A_1 \cdot \ldots \cdot A_{D}}(\omega,1)| \> \geq \> h^{D/2} 
$$
with positive probability.
It follows that, if $h^{D/2} >n$, then (\ref{eq:explcondition}) must fail at some $i<D$. This means that
$$
|\Sf_{A}(\omega,D)| 
\> \geq \>
|\Sf_{A_1 \cdot \ldots \cdot A_{D}}(\omega,1)|>\frac{n}{h^{2/\varepsilon}} 
\> \geq \>
\frac{n}{k^{2/\varepsilon}} 
$$
with probability at least
\begin{equation}\label{eq:probsph}
\left( 1 - \frac{2}{h^{2\varepsilon^{-1}-1}} \right)^{hD} 
\> \geq \> 
1 - \frac{2D}{h^{2\varepsilon^{-1} -2}} .
\end{equation}

Finally, we check that, for some $C(\varepsilon) =O(\varepsilon^{-1})$, the choice
\begin{align*}
D & =\left\lfloor \frac{C \log n}{\log k}\right\rfloor, & h & =\left\lfloor\frac{k}{D}\right\rfloor,
\end{align*}
does the job. This choice optimizes the number of steps $D$, as one could not hope to take $D$ smaller because $\left|\Sf_A\left(\omega,o(\log_{k}n)\right)\right|\leq k^{o(\log_{k}n)}=n^{o(1)}$; now we need to verify that this $D$ is sufficient, i.e. that $h^{D/2}>n$ and \eqref{eq:probsph} goes to $1$.
We omit the floor notation since it is not crucial.

To check that $h^{D/2} > n$, some computations show that this is equivalent to
\begin{equation} \label{eqc}
\left( \frac{k \log k}{C \log n} \right)^C 
\> > \>
 k^2 .
\end{equation} 
From the hypotheses we have $\log n \leq k^{1/(1+\varepsilon)}$, so a polynomial portion of $k$ survives on the left side of (\ref{eqc}).
When $n$ and so $k$ are sufficiently large, the existence of a good $C =O(\varepsilon^{-1})$ is assured.
Similarly, the probability in (\ref{eq:probsph}) is bounded from below by
$$
1-\frac{2D}{h^{2\varepsilon^{-1} -2}} 
\> \geq \>
 1- \left( \frac{2C \log n}{\log k} \right) \left( \frac{C \log n}{k \log k} \right)^{2\varepsilon^{-1} -2} .
$$
Since $k \geq (\log n)^{1+\varepsilon}$, we have
$$ \left( \frac{2C \log n}{\log k} \right) \left( \frac{C \log n}{k \log k} \right)^{2\varepsilon^{-1} -2} 
\> \leq \> 
\frac{2C^{2\varepsilon^{-1}-1}}{(\log n)^{1-2\varepsilon}} , $$
which is going to zero as $n$ grows.
\end{proof}

To obtain Proposition \ref{pr1} from Lemma \ref{lemGS},
we have to close the gap between $k^{2/\varepsilon}$ and $k^{\varepsilon/2}$.
We apply Corollary \ref{co:Explicit} a few more times.

\begin{proof}[Proof of Proposition \ref{pr1}]
It is sufficient to prove the result for small $\varepsilon$.
Moreover, we omit the floor notation since it is not crucial.
Let $C_2= 6 \varepsilon^{-3}$.
For $n$ (and thus $k$) large enough, we have $k \geq \frac{k}{2} + C_2 k^{\varepsilon^{2}}$,
so we can rewrite $A \subseteq G$ as
\begin{align*}
A & =\bigcup_{i=0}^{C_2}A_{i}\cup A', & A_{0} & \sim\mu_{G}\left(\frac{k}{2}\right), \\
A_{i} & \sim\mu_{G}( k^{\varepsilon^{2}}) \ \ \ (1\leq i\leq C_2), & A' & \sim\mu_{G} \left( \frac{k}{2} - C_2 k^{\varepsilon^{2}} \right).
\end{align*}
From $k \geq (\log n)^{1+\varepsilon}$ we obtain $\frac{k}{2} \geq (\log n)^{1+\frac{\varepsilon}{2}}$ for large $n$,
so up to replacing $\varepsilon$ we can use Lemma~\ref{lemGS} with $A_{0}$, implying that
$$ \left|\Sf_{A_{0}}\left(\omega, C_1 \log_{k}n\right)\right| 
\> \geq \>
 \frac{n}{(k/2)^{2/\varepsilon}} 
 \> \geq \>
  \frac{n}{k^{2/\varepsilon}}  $$
holds
with probability $1-o(1)$ as $n\rightarrow\infty$
(here $C_1$ is the constant from Lemma \ref{lemGS}).
At this point we apply Corollary~\ref{co:Explicit} iteratively with each $A_{i}$,
with $k^{\varepsilon^2}$ instead of $k$, and $d= (2\varepsilon)^{-1}$,
starting at the first step with $X=\Sf_{A_{0}}(\omega,C_1\log_k n)$.
If (\ref{eq:explcondition}) were true for all $1\leq i\leq C_2$, then we would have
$$ |X^{A_{1}\cdot\ldots\cdot A_{C_2}}|
\> \geq \>
\frac{n}{k^{2/\varepsilon}}\cdot (k^{\varepsilon^{2}})^{C_2/2} 
\> \geq \>
n k^{1/\varepsilon} 
\> > \>
n, $$
with positive probability, which is impossible.
Therefore, (\ref{eq:explcondition}) fails for some $i\leq C_2$.
We conclude that
\begin{align*}
	\left|\Sf_A\left(\omega,(C_1+C_2)\log_{k}n \right)\right| & \> \geq \> 
	|X^{A_{1}\cdot\ldots\cdot A_{C_2}}| \\ & \> \geq \> 
	\frac{n}{ (k^{\varepsilon^2})^{(\varepsilon^{-1}/2)}} \\ & \> = \> 
  	\frac{n}{k^{\varepsilon/2}} \> ,
	\end{align*} 
with probability that, up to a multiplicative constant converging to $1$, is bounded from below by
$$
\left(1-\frac{2}{(k^{\varepsilon^2})^{(2\varepsilon)^{-1}-1}} \right)^{C_2 k^{\varepsilon^{2}}}
\> \geq \>
 \left(1-\frac{2}{k^{\varepsilon/2}} \right)^{6 \varepsilon^{-3} k^{\varepsilon^{2}}} 
\> \geq \>
1 - \frac{12 \varepsilon^{-3}}{k^{ \frac{\varepsilon}{2} - \varepsilon^2}} .
$$
This is going to $1$ as $k$ goes to infinity, and the proof is complete.
\end{proof}

\begin{remark}\label{re:cepsilon}
Our proof of Proposition \ref{pr1} gives $C=O(\varepsilon^{-3})$ for small $\varepsilon$,
but with more work the same argument provides $O(\varepsilon^{-1-\delta})$ for all fixed $\delta>0$: one needs to apply Corollary~\ref{co:Explicit} in the proof fewer times, initially with larger steps $k^{\varepsilon^{\delta}}$ rather than $k^{\varepsilon^{2}}$, and then with increasingly smaller steps until we reach the desired sphere size.
\end{remark}

\vspace{0.1cm}
\section{Bounding the diameter} \label{se:fill}

In this section we study the situation where a set $X \subseteq \Omega$ (typically a sphere), is large.
The first observation is that we can reach the whole $\Omega$ with the help of a few random elements.

\begin{lemma} \label{le:fill}
Let $G$ be any finite group acting transitively on a set $\Omega$ with $|\Omega|=n$.
Let $X\subseteq\Omega$ with $|X|\geq n/m$ for some $m \geq 1$, and let $k \geq 4m \log n$.
Then, for $B\sim \mu_{G}(k)$, we have $X^B=\Omega$
with probability $1-o(1)$ as $n\to\infty$.
More precisely,
$$
\mathbb{P}_{\mu_G(k)}(X^{B}=\Omega)
\> \geq \>
1-\frac{1}{2^{k/m}} .
$$
\end{lemma}
\begin{proof}
Let $t$ be the size of a stabilizer, and fix at first any $\omega \in \Omega$.
For each $x \in X$, there are precisely $t$ elements $g\in G$ such that $x^{g}=\omega$.
Thus, for a random $g\in G$, the probability that $\omega \notin X^{g}$ is
$$ 1 - \frac{t|X|}{|G|} 
\> = \> 
1-\frac{|X|}{n}
\> \leq \>
 1-\frac{1}{m} . $$
For $B\sim \mu_{G}(k)$, we have
$$ \mathbb{P}_{\mu_G(k)}(\omega \notin X^{B})
 \> \leq \> \left(1-\frac{1}{m}\right)^{k}
\> \leq \> \frac{1}{e^{k/m}}. $$
Summing over all $\omega \in\Omega$, we obtain
\begin{align*}
	\mathbb{P}_{\mu_G(k)}(X^{B}=\Omega) & \> = \> 
	1 - \mathbb{P}_{\mu_G(k)}(\omega \notin X^{B} \text{ for at least one } \omega \in\Omega) \\ & \> \geq \> 
  	1- \frac{n}{e^{k/m}} . \> 
	\end{align*} 
The proof follows from the lower bound on $k$, because
\begin{equation*}
n 
\> \leq \>
 e^{k/(4m)} 
 \> < \>
  (e/2)^{k/m} . \qedhere
\end{equation*}
\end{proof} 

One observation should be made at this point.
The graphs we are working with are directed and not vertex-transitive in general,
therefore the fact that a certain sphere with a fixed base point fills $\Omega$, does not immediately give a bound to the diameter.
Looking at (\ref{eqDiam}), summing over all $\omega \in \Omega$ is too expensive.
We overcome this problem by applying the previous argument twice,
exploiting the fact that two uniformly random multisets $A$ and $A^{-1}$ of $G$ have the same distribution as random variables.

\begin{lemma} \label{lemFil2} 
Suppose that, for some $\omega \in \Omega$,
for $A \sim \mu_G(k)$ we have $\Sf_A(\omega,t)=\Omega$
with probability $1-o(1)$ as $n\to\infty$.
Then, for $B \sim \mu_G(2k)$, we have
$$ \diam( \> \Sch(G \circlearrowleft \Omega, B) \> ) 
\> \leq \> 
2t $$
with probability $1-o(1)$ as $n\to\infty$.
\end{lemma}
\begin{proof}
Let $A_1, A_2 \sim \mu_G(k)$.
From the hypotheses, and the observation above, we have
$$
\Sf_{A_1}(\omega,t) \> = \> \Sf_{A_2^{-1}}(\omega,t) \> = \> \Omega
$$
with probability $1-o(1)$ as $n\to\infty$.
Now let $x,y \in \Omega$.
With probability $1-o(1)$,
there is $g_1 \in (A_1)^t$ such that $\omega^{g_1} = y$,
and also there is $g_2 \in (A_2^{-1})^t$ such that $\omega^{g_2} = x$.
In particular, 
with probability $(1-o(1))^2=1-o(1)$,
there exists an element $g_2^{-1} g_1 \in (A_1 \cup A_2)^{2t}$ such that
$$
 x^{g_2^{-1} g_1} \> = \> \omega^{g_1} \> = \> y .
$$
This concludes the proof.
\end{proof}

Proposition \ref{pr2} now follows easily.

\begin{proof}[Proof of Proposition \ref{pr2}]
By an application of Lemma \ref{le:fill} with $m=k^\delta$, for $B \sim \mu(k+4k^\delta \log n)$,
we have $\Sf_B(\omega,t+1) =\Omega$
with probability $1-o(1)$ as $n\to\infty$.
The proof follows from Lemma \ref{lemFil2}.
\end{proof}

As suggested by S.~Eberhard, it is possible to prove Proposition \ref{pr2} in a slightly different way:
up to changing the constant in the bound for $k$ in Lemma~\ref{le:fill}, we can achieve probability $1-o(n^{-1})$ therein,
and then using a union bound $\Sf_A(\omega,t)=\Omega$ for all $\omega \in \Omega$ with probability $1-o(1)$.

We are ready to conclude the proof of the main theorem.
We observe that, if $x \in (0,1)$, then $\frac{1}{1+x}  <1 - \frac{x}{2}$.

\begin{proof}[Proof of Theorem \ref{th:main}]
First, we observe that it is enough to prove the result when $\varepsilon \in (0,1)$.
From $k \geq (\log n)^{1+\varepsilon}$ we obtain $\frac{k}{4} \geq (\log n)^{1+\frac{\varepsilon}{2}}$ for large $n$,
so up to replacing $\varepsilon$ we can apply Proposition \ref{pr1} with $\frac{k}{4}$ elements.
For $A \sim \mu_G(k/4)$, $t=O_\varepsilon(\log_k n)$, we have
$$ \left|\Sf_A \left(\omega,t\right)\right| 
\> \geq \> \frac{n}{(k/4)^{\varepsilon/2}} 
\> \geq \> \frac{n}{k^{\varepsilon/2}} $$
with probability $1-o(1)$ as $n\to\infty$.
We now apply Proposition \ref{pr2} with $\delta = \frac{\varepsilon}{2}$,
so that we need $2(k/4) + 8 (k/4)^{\varepsilon/2} \log n$ random elements in total.
It remains to check that these are fewer than the $k$ random elements actually available.
Since $k \geq (\log n)^{1+\varepsilon}$, we have
$$ 2(k/4) + 8 (k/4)^{\varepsilon/2} \log n
\> \leq \>
 \frac{k}{2} + \frac{k^{ \frac{\varepsilon}{2} + \frac{1}{1+\varepsilon}}}{4^{\varepsilon/2}/8} 
 \> < \> 
 k $$
when $k$ is sufficiently large.
The proof is complete.
\end{proof}

\vspace{0.1cm}
\section*{Acknowledgements}
We thank the referees for their useful comments.
The first author was funded by a Young Researcher Fellowship from the HUN-REN Alfr\'ed R\'enyi Institute of Mathematics,
and by the Leibniz Fellowship 2405p from the Mathematisches Forschungsinstitut Oberwolfach (MFO).
The second author was supported by the Royal Society.

	\vspace{0.1cm}


\vspace{0.1cm}
\thebibliography{10}

\bibitem{ABM87} N. Alon, A. Barak, U. Manber,
   \textit{On disseminating information reliably without broadcasting},
	in Proc. Seventh Internat. Conf. on Distributed Computing Systems (ICDS) (1987), 74-81.
	
	\bibitem{AR94} N. Alon, Y. Roichman,
  \textit{Random Cayley graphs and expanders},
	Random Structures and Algorithms {\bfseries 5(2)} (1994), 271-284.
	
	\bibitem{Bab06} L. Babai,
   \textit{On the diameter of Eulerian orientations of graphs},
	in Proc. 17th Ann. Symp. on Discr. Alg. (SODA) (2006), 822-831.

\bibitem{BF82} B. Bollob\'as, W. Fernandez De La Vega,
   \textit{The diameter of random regular graphs},
	Combinatorica \textbf{2(2)} (1982), 125-134.
	
	\bibitem{BGGT15} E. Breuillard, B. Green, B. Guralnick, T. Tao,
  \textit{Expansion in finite simple groups of Lie type},
	Journal of the European Mathematical Society {\bfseries 17} (2015), 1367-1434.

\bibitem{DH96} C. Dou, M. Hildebrand,
   \textit{Enumeration and random random walks on finite groups},
	Annals of Probability \textbf{24(2)} (1996), 987-1000.
	
	\bibitem{EJ22} S. Eberhard, U. Jezernik,
   \textit{Babai's conjecture for high-rank classical groups with random generators},
	Inventiones Mathematicae \textbf{227} (2022), 149-210.
	
	\bibitem{FJRST98} J. Friedman, A. Joux, Y. Roichman, J. Stern, J.P. Tillich,
   \textit{The action of a few permutations on $r$-tuples is quickly transitive},
	 Random Structures and Algorithms \textbf{12(4)} (1998), 335-350.
	 
	 \bibitem{Gro77} J.L. Gross,
   \textit{Every connected regular graph of even degree is a Schreier coset graph},
	Journal of Combinatorial Theory (B) \textbf{22(3)} (1977), 227-232.
	
	\bibitem{Hel15} H.A. Helfgott,
   \textit{Growth in groups: ideas and perspectives},
	Bulletin of the American Mathematical Society \textbf{52(3)} (2015), 357-413.
	
	\bibitem{KN88} L.G. Kov\'acs, M.F. Newman,
   \textit{Generating transitive permutation groups},
	Quarterly Journal of Mathematics \textbf{39} (1988), 361-372.
	
	\bibitem{LMM00} A. Lucchini, F. Menegazzo, M. Morigi,
   \textit{Asymptotic results for transitive permutation groups},
	Bulletin of the London Mathematical Society \textbf{32(2)} (2000), 191-195.
	
	\bibitem{Mur19} B. Murphy,
   \textit{Group action combinatorics},
	available at \texttt{https://arxiv.org/pdf/1907.13569} (2019).
	
	\bibitem{Roi96} Y. Roichman,
   \textit{On random random walks},
	Annals of Probability \textbf{24(2)} (1996), 1001-1011.
	
	\bibitem{Sab22} L. Sabatini,
   \textit{Random Schreier graphs and expanders},
	Journal of Algebraic Combinatorics \textbf{56(3)} (2022), 889-901.
	
	\bibitem{Skr23} S.V. Skresanov,
   \textit{On directed and undirected diameters of vertex-transitive graphs},
	Combinatorica \textbf{44} (2024), 1353-1366.
\end{document}